\documentclass[10pt]{amsart}

\usepackage{amssymb,amsthm,amsmath}

\usepackage{url}

\usepackage{color}
\usepackage{verbatim}
\theoremstyle{plain}

\newtheorem{theorem}{Theorem}[section]

\newtheorem{proposition}[theorem]{Proposition}

\theoremstyle{definition}

\newtheorem{definition}[theorem]{Definition}

\newtheorem{remark}[theorem]{Remark}

\newcommand{\R}{{\mathbb R}}

\newcommand{\eps}{\varepsilon}

\newcommand{\sgn}{\text{sgn}}

\newcommand{\PP}{\mathbb{P}} % for probabilities
\newcommand{\EE}{\mathbb{E}} % for expectation

\begin{document}  % ========== BEGIN DOCUMENT

\title{On  Khintchine type inequalities for $k$-wise independent Rademacher random variables}
\footnote{BP is pleased to acknowledge the support of a University of Alberta start-up grant and National Sciences and Engineering Research Council of Canada Discovery Grant number 412779-2012.}
\author{Brendan Pass and Susanna Spektor}
\date{}

\address{Brendan Pass 
\noindent Address: University of Alberta, Edmonton, AB, Canada, T6G2G1}
\email{pass@ualberta.ca}

\address{Susanna Spektor
	\noindent Address: Sheridan college Institute of Technology and Advanced Learning, 4180 Duke of York Blvd., Mississauga, ON, Canada, L5B0G1}
\email{susanna.spektor@sheridancollge.ca}

%\date{}

\maketitle

\begin{abstract}
%We prove a Khintchine type inequality under the  assumption that the
%Rademacher random variables are k-dependent,when $k=2$ or $3$.
We consider Khintchine type inequalities on the $p$-th moments of vectors of $N$  $k$-wise independent Rademacher random variables.   We show that an analogue of Khintchine's inequality holds, with a constant $N^{1/2-k/2p}$, when $k$ is even.  We then show that this result is sharp for $k=2$; in particular, a version of Khintchine's inequality  for sequences of pairwise Rademacher variables \emph{cannot} hold with a constant independent of $N$.   We also characterize the cases of equality and show that, although the vector achieving equality is not unique, it is unique (up to law) among the smaller class of exchangable vectors of pairwise independent  Rademacher random variables. As a fortunate consequence of our work, we obtain similar results for $3$-wise independent vectors.

 %We establish that an analogue of Khintchine's inequality cannot hold in this setting with a constant that is independent of $N$; in fact, we prove that the best constant one can hope for is \emph{at least} $N^{1/2-1/p}$.  Furthermore, we show that this estimate is \emph{sharp} for exchangeable vectors when $p=4$.

\medskip

\noindent 2010 Classification: 46B06, 60E15
% 52A23,  46B09; Secondary
%

\noindent Keywords:
Khintchine inequality,  Rademacher random variables,
$k$-wise independent random variables.
\end{abstract}

%\thispagestyle{empty}

%\begin{abstract}

%\end{abstract}

\setcounter{page}{1}

\section{Introduction}

This short note concerns Khintchine's inequality, a classical theorem in probability, with many important applications in both probability and analysis
(see \cite{Garlin, Kah, LT, MS, PShir} among others). It states that  the $L_p$ norm of the weighted sum of
independent Rademacher random variables is controlled by its $L_2$ norm; a precise statement follows.
We say that $\eps _0$ is a Rademacher random variable if
$\PP(\varepsilon_0=1)=\PP(\varepsilon_0=-1)=\displaystyle{\tfrac 12}$.
Let $\bar \varepsilon_i$, $1\leq i\leq N$,  be independent
copies of $\varepsilon_0$ and $a \in \R^{N}$.
Khintchine's inequality
(see, for example, Theorem 2.b.3 in \cite{LT} , Theorem 12.3.1 in \cite{Garlin} or the original work of Khintchine \cite{Khintchine23})
states that, for any $p > 0$
\begin{align}\label{1}
 B(p) \left(\EE\left|\sum_{i=1}^{N}a_i\bar \varepsilon_i\right|^2\right)^{\frac 12}=B(p)||a||_2\leq \left( \EE\left|\sum_{i=1}^{N}a_i\bar \varepsilon_i\right|^p\right)^{\frac 1p}\leq
 C(p) \, \|a\|_{2}=C(p) \left(\EE\left|\sum_{i=1}^{N}a_i\bar \varepsilon_i\right|^2\right)^{\frac 12}.
\end{align}

%Note that the (Rademacher) random vector $\varepsilon = (\varepsilon_1, \ldots, \varepsilon_{N})$
%in the Khintchine inequality has independent coordinates.
%Iin many problems of Analysis and Probability it is important to consider random vectors with dependent coordinates, e.g.  so-called
%og-concave random vectors, which in general have dependent coordinates, but
 %whose behaviour is similar to that of Rademacher random vector or to the Gaussian random vector
We will mostly be interested in the upper Khintchine inequality; that is, the second inequality in \eqref{1}. Note here that the upper constant $C(p)$ depends only on $p$; in particular, it does not depend on $N$.  In what follows, we take $C(p)$ to be the best possible constant in \eqref{1}.  This value is in fact known explicitly \cite{H}:
\[ C(p) = \left\{
  \begin{array}{l l}
   1 & \quad0<p\leq 2\\
    \frac{\sqrt{2}\Gamma(\frac{p+1}{2})}{\sqrt{\pi}^{1/p}} & \quad p>2.
  \end{array} \right.\]

 It is natural to ask whether the independence condition can be relaxed; indeed, random vectors with dependent coordinates arise in many problems in probability and analysis (see e.g. \cite{G} and the references therein).  In this short paper, we are interested in what can be said when the independence assumption on the coordinates is relaxed to pairwise (or, more generally, $k$-wise)  independence.  %$k$-dependent Rademacher random variables.

\begin{definition} We call an $N$-tuple $\varepsilon=\{ \varepsilon_i\}_{i= 1}^N$ of Rademacher random variables a \textit{Rademacher vector}, or (finite) \textit{Rademacher sequence.}
For a fixed non-negative integer $k$, a Rademacher vector is called \emph{$k$-wise independent} if  any  subset  $\{\varepsilon_{i_1}, \varepsilon_{i_2}, \ldots, \varepsilon_{i_k}\}$ of size $k$ is mutually independent.

%Let $\varepsilon_1, \varepsilon_2, \ldots,$ be a sequence of Rademacher random variables defined on a common probability space $(\Omega, F, \PP)$. A Rademacher sequence is pairwise independent if and only if for each $i>0$ $\displaystyle{\PP(\varepsilon_1=\varepsilon_i=+1)=\frac 14}$ (which is equivalent to their orthogonality in the Hilbert space $\mathfrak{L}^2(\Omega, F, \PP)$). It is $k$-dependent  (for fixed $k\geq 1$) if and only if for each $m, n \in \N$ and for each choice of values $x_j=+1, \, 1\leq j\leq m, \, m+k+1\leq j \leq m+k+n$,
%\begin{align*}
%&\PP(\varepsilon_1=x_1, \ldots, \varepsilon_m=x_m, \varepsilon_{m+k+1}=x_{m+k+1}, \ldots, \varepsilon_{m+k+m}=x_{m+k+n})\\
%&=\PP(\varepsilon_1=x_1, \ldots, \varepsilon_m=x_m)\PP(\varepsilon_{m+k+1}=x_{m+k+1}, \ldots, \varepsilon_{m+k+m}=x_{m+k+n}).
%\end{align*}
\end{definition}

When $k=2$ in the preceding definition, we will often use the terminology \emph{pairwise independent} in place of $2$-wise independent.  For more on $k$-wise independent sequences and their construction, see, for example \cite{D1, R1, R2}.

As it will be useful in what follows, we note that instead of random variables, it is equivalent to consider probability measures $P$ on the set $\{-1,1\}^N$, where $P=law(\varepsilon)$.  The condition that $\varepsilon$ is a Rademacher vector is then equivalent to the condition that the projections $law(\varepsilon_i)$ of $P$ onto each copy of $\{-1,1\}$ are all equal to $P_1:=\frac{1}{2}[\delta_{-1} +\delta_1]$.  The $k$-wise independence condition is equivalent to the condition that the projections $law(\varepsilon_{i_1},\ldots,\varepsilon_{i_k})$  of $P$ onto each $k$-fold product $\{-1,1\}^k$ is product measure $\otimes^kP_1$.

An interesting general line of research in probability aims to understand which of the many known properties of mutually independent sequences carry over to the $k$-wise independent setting; how much independence is actually needed to assert various properties? Some results,  including the second Borel-Cantelli lemma and the strong law of large numbers (see, for instance, \cite{Etemadi81} and \cite{law}) hold true for pairwise independent sequences, whereas others, such as the central limit theorem, do not.
%In general, there are many known results on mutually  independent sequences, and it is natural to ask whether these carry over to $k$-wise independent sequences, for $k \geq 2$; in other words, how much independence is actually needed to assert various properties? The answer depends on the result in question; some theorems extend to the pairwise independent setting,  including the second Borel-Cantelli lemma and the strong law of large numbers (see e.x. \cite{law}), whereas others, such as the central limit theorem, do not.  %Contrary to the usual opinion that the mutual independence is a natural property, it is an extremely particular case in the wider class of $k$-dependent random variables.
%\marginpar{\red{I don't think Borel-Cantelli needs even $2$-dependence.  I'm not aware of a version of the law of large numbers without mutual independence - can you provide a reference?}}
 %In general, description of this class is not known.
We found it surprising that little seems to be known about Khintchine's inequality for $k$-wise independent sequences (except when $k \geq p$, as we discuss briefly below).

It is therefore natural to ask whether Khintchine's inequality holds for $k$-wise independent Rademacher random variables, and, if not, to understand how badly it fails.  More precisely, we  focus on the upper Khintchine inequality and define
%\begin{question}
%Does Khintchine's inequality hold for $k$-wise independent random variables?  That is, is there a constant $C(p,k)$, independent of $N$, such that, for any $k$-wise independent sequence of length $N$,

%$$
 %\left(\EE\left|\sum_{i=1}^{N}a_i\varepsilon_i\right|^p\right)^{\frac 1p}\leq
 %C(p,k) \, \|a\|_{2}.
%$$
%f the answer is no, how badly does the best constant in the preceding inequality grow with $N$?
%Can we relax independence of Rademacher random variables in classical Khintchine inequality (\ref{1}) to $k$-dependence to obtained the same estimate, i.e. the constant $C(p)$ independent of $N$? If not, for some $k$, how bad C(N, p) is?
%\end{question}

%\marginpar{{\red I'm not sure whether it's better to formulate the problem as a minimization over exchangeable sequences.  Lemma 2.3 and Corlallary 2.4 require exchangability, but the other results don't.}}
\begin{equation}\label{constant}
C(N,p,k) = \sup_{\substack{a\in \mathbb{R}^N: ||a||_2 =1 \\ \varepsilon \text{ is a k -wise independent Rademacher vector}} }\left(\EE \left|\sum_{i=1}^N a_i\varepsilon_i\right|^p\right)^{1/p}.
\end{equation}
%\begin{remark}
%For a fixed $a$, the maximization above is a linear program, which is fact similar to a discrete version of the multi-marginal optimal transport problem.  The difference is, here the $k$-fold marginal $P_k$ are prescribed rather than the single marginal %$P_1$; that is, the discrete multi-marginal optimal transport problems is the case $k=1$.
%\end{remark}
%the questions above can then be formulated, respectively, as:
The questions we are interested in can then be formulated as:
\begin{enumerate}
\item[1.] Is $C(N,p,k)$ bounded as $N \rightarrow \infty$, for a fixed $p >k$ ?
\item[2.] If not, what is the growth rate of $C(N,p,k)$?
\end{enumerate}
Note that the $C(N,p,k)$ form a monotone decreasing sequence in $k$, as the $k$-dependence constraint becomes increasingly stringent as $k$ grows. % We define $C(N,p,\infty)$ to be the best constant in Khintchine's inequality (for independent random %variables):
%$$
%C(N,p,\infty) = \sup_{a\in \mathbb{R}^N: ||a||_2 =1 }\left(\EE\left|\sum_{i=1}^N a_i\bar \epsilon_i\right|^p\right)^{1/p}.
%$$
%where the $\bar \varepsilon_i$ are mutually independent Rademacher random variables.  %where $I$ denotes the independent probability measure on $\{-1,1\}^N$.
Note that, as mutual independence implies $k$-wise independence for any $k$, we have $C(N,p,k) \geq C(p)$, where $C(p)$ is the best constant in the classical Khintchine inequality \eqref{1}. % In this notation, the classical Khintchine inequality means that $C(N,p,\infty)$ is bounded as $N$ goes to $\infty$ for each fixed $p$.

Some properties of $C(N,p,k)$ are easily discerned. For example, it is straightforward to see that $C(N,2,k) =1$.   \begin{comment}Furthermore, by an application of the Cauchy-Schwartz inequality, we get, for any Rademacher $\varepsilon$ and any $a$,

\begin{equation}\label{holderbound}
 \left(\EE\left|\sum_{i=1}^N a_i\varepsilon_i\right|^p\right)^\frac{1}{p} \leq \sqrt{N} ||a||_2
\end{equation}
and so we  have
$$
C(N,p,k) \leq \sqrt{N}.
$$
In fact, if $a=\frac{1}{\sqrt{N}}(1,1,...,1)$, then for a random vector with $law(\varepsilon)=\frac{1}{2}[\delta_{1,1,1,\ldots,1} +\delta_{-1,-1,-1,\ldots,-1} ]$, we get \emph{equality} in \eqref{holderbound}.  This $\epsilon$ is $1$-wise independent, which simply means that it has Rademacher marginals and so we have $C(N,p,1) =\sqrt{N}$.  Clearly, this vector is not pairwise independent, however, and so it provides no further information on $C(N,p,k)$ for $k \geq 2$.
\end{comment}
Let us also mention that, when $p$ is an even integer, and $k  \geq p$, it is actually a straightforward calculation to show that $C(N,p,k) =C(p) $ is independent of $N$ (that is, Khintchine's inequality for $k$-wise independent random variables holds with the same constant as in the independence case). 
% This seems to be a "folklore" result, which is well known to experts, but for which we were unable to find a suitable reference.  

For $k <p$ and even, we first show that $C(N,p,k) \leq  C(k)^{k/p} N^{1/2 -k/2p}$, by combining a standard interpolation argument with the classical Khintchine inequality and the observation above.  This provides some information on the second question above for general $k$.

We then focus on the $k=2$ case.  We prove that for $p \geq 2$ and $N$ even, $C(N,p,2) = N^{1/2-1/p}$, providing a negative answer to the first question above.   We construct an explicit pairwise independent Rademacher sequence satisfying the equality.  Finally, we characterize the cases of equality, and prove that although this equality may be achieved by multiple Rademacher vectors, the one we construct is the unique exchangeable equality case (up to law).
% Moreover, if we define $C_e(N,p,k)$ as in \eqref{constant}, but with the supremum restricted to \emph{exchangeable} %Rademacher vectors $\varepsilon$,  and consider the $p=4$ case, we prove that $C_e(N,4,2) = N^{1/2-1/4} =N^{1/4}$.
\begin{comment}Our proof is elementary.  We fix the vector $a=\frac{1}{\sqrt{N}}(1,1,...,1)$ in \eqref{constant}, and for each $N$, explicitly solve the  linear program corresponding to maximizing the expression in \eqref{constant} over all pairwise independent Rademacher vectors; this calculation shows $C(N,p,2) \geq N^{1/2-1/p}$ (see Theorem \ref{main} below).  We are then able to show, at least for exchangeable vectors when $p=4$, that this choice of $a$ yields the worst case in the maximization \eqref{constant}, implying  $C_e(N,4,2) =N^{1/4}$ (see Lemma \ref{worstcase} and Corollary \ref{exchangable} below).  
\end{comment}

As a fortunate consequence of our work here, we obtain analogous results for $k=3$. Understanding the $k \geq 4$ case remains an interesting open question.  % It is possible that our method of proof may be extended to higher $k$; the main challenge is in solving explicitly the linear program \eqref{constant} for a fixed $a$ to yield the maximizing $k$-wise independent Rademacher vector.  

%Finally, we note that the maximization arising from fixing $a$ in \eqref{constant}  is similar to a discrete version of the multi-marginal Monge-Kantorovich problem (see \cite{P} for a survey); the difference is that here the $k$-fold marginals $P_k =\otimes^kP_1$ are prescribed, whereas in the Monge-Kantorovich problem only the $1$ fold marginals $P_1$  are fixed. It is %perhaps  worth noting that studying  a continuous version of this problem (that is, a version where the marginals $P_1$ of the $\epsilon_i$ are allowed to be more general measures) may reveal insights into $k$-wise dependent versions of the Marcinkiewicz - Zygmund  inequality.

\section{A general estimate on $C(N,p,k)$}
We begin by establishing an upper bound on $C(N,p,k)$ via a straightforward interpolation argument.
\begin{proposition}\label{prop: general bound}
For all $p\geq k \geq 2$ and $k$ even, we have $C(N,p,k) \leq C(k)^{k/p} N^{1/2 -k/2p}$.
\end{proposition}
\begin{proof}
Let $\epsilon=(\epsilon_1,...,\epsilon_N)$ be a $k$-wise independent Rademacher vector of length $N$, and $a=(a_1,...,a_N) \in \mathbb{R}^N$.  Set $f=|\sum_{i=1}^N a_i\epsilon_i|$, so that $f$ is a function on the underlying probability space.  Writing $f^p= f^kf^{p-k}$, we apply Holder's inequality to get
\begin{equation}\label{eqn: holder}
||f||_{p} \leq ||f||_k^{k/p} ||f^{p-k}||_{\infty}^{1/p}
\end{equation}
Now, note that the multinomial theorem implies

$$
||f||_k^k =\mathbb{ E}[|\sum a_i\epsilon_i|^k] =\sum_{j_1+j_2+...+j_m =k} \frac{k!}{j_1!j_2!....j_m!} \mathbb E\Big(\Pi_{i=1}^m (a_i\epsilon_i)^{j_i}\Big),
$$ 
As each term $\Pi_{i=1}^m (a_i\epsilon_i)^{j_i}$ contains only at most $k$ distinct $\epsilon_i$, and the distribution is $k$-wise independent, the expected value is identical to what we would obtain from a mutually independent series, $\bar \epsilon$, and so we have, by the classical Khintchine inequality,
$$
||f||_k^k =\mathbb E[|\sum a_i\epsilon_i|^k] =\mathbb E[|\sum a_i\bar \epsilon_i|^k] \leq C(k)^k||a||_2^k.
$$ 

On the other hand,

\begin{equation}\label{eqn:c-s}
||f^{p-k}||_{\infty}^{1/p} \leq ||a||_1^{1-k/p} \leq [||a||_2 N^{1/2} ]^{1-k/p}=N^{1/2-k/2p}||a||_2^{1-k/p}
\end{equation}
where the last inequality is by Cauchy-Schwartz. Combining these, we have

$$
||f||_{p} \leq  C(k)^{k/p}||a||_2^{k/p}N^{1/2-k/2p}||a||_2^{1-k/p}= C(k)^{k/p}N^{1/2-k/2p}||a||_2
$$
which implies the desired result.
\end{proof}
%By the following theorem we answer this question in the case when $k=2$ and when $a=(1, 1, \ldots, 1)$.
\section{Pairwise Independence: the precise value of $C(N,p,2)$.}
The following Theorem provides a negative  answer to the first question in the introduction when $k=2$, and provides the precise answer to the second question in the same case.  In particular, it shows that the bound  $C(N,p,2) \leq N^{1/2 -1/p}$ from the previous proposition is sharp for even $N$.
\begin{theorem}\label{main}
Let $N=2n$ be even.  There exists a pairwise independent Rademacher vector $\varepsilon=(\varepsilon_1,...,\varepsilon_N)$  such that, for  $a=(1,1,...,1) \in \mathbb{R}^N$  and  all $p \geq 2$,

\begin{equation}\label{theorem}
\left(\EE \left|\sum_{i=1}^N a_i\varepsilon_i\right|^p\right)^{1/p}= N^{1/2-1/p} \|a\|_2
\end{equation}

Consequently, $C(N, p, 2) = N^{1/2-1/p}$. 

%Furthermore, the optimal $a$ and $\epsilon$ are unique in the following sense: if $a$ and $\varepsilon'=(\varepsilon_1',...%\varepsilon_N')$ satisfy $\left(\EE \left|\sum_{i=1}^N a_i\varepsil%n_i'\right|^p\right)^{1/p}= N^{1/2-1/p} \|a'\|_2$, we %must have $a'=c'(1,.....1)$ for some constant $c'$ and $law(\varepsilon ') =law (\epsilon)$.
\end{theorem}
\begin{proof}
%We explicitly construct a pairwise independent Rademacher vector $\varepsilon=(\varepsilon_1,...\varepsilon_N)$  for which we get equality in \eqref{theorem}, and then show that this maximizes the left hand side over the set of all such vectors. 
We will construct the probability measure $P=law(\varepsilon)$.

We define $P=\frac{1}{N} P_a+ \frac{N-1}{N} P_b$ where $P_a=\frac{1}{2}[\delta_{1,1,1,\ldots,1} +\delta_{-1,-1,-1,\ldots,-1}]$ is uniform measure on the two points $(1,1,1,\ldots,1), (-1,-1,\ldots,-1) \in \{-1,1\}^N$ and $P_b$ is uniform measure on the set of all points with an equal number of $1$'s and $-1$'s; that is, points which are permutations of $\{{\underbrace{1, 1, \ldots, 1}_{N/2 \, \textit{of them}}\underbrace{-1, -1, \ldots, -1}_{N/2 \, \textit{of them}}}\}$.
We first verify that this is a pairwise independent probability measure; that is, that it's twofold marginals are $\frac{1}{4}(\delta_{1,1} +\delta_{1,-1} +\delta_{-1,1} +\delta_{-1,-1})$.  By symmetry between the coordinates, it suffices to verify this fact for the projection $P_2$ on the first two copies of $\{-1,1\}$.  We have

$$
P_2(1,1) = \frac{1}{N}P_a(1,1,1,\ldots,1) +\frac{N-1}{N}P_b\{\varepsilon:(\varepsilon_1,\varepsilon_2)=(1,1)\}.
$$
Now, $P_a(1,1,1,\ldots,1) =\frac{1}{2}$, and it  is easy to see that $P_b\{\varepsilon:(\varepsilon_1,\varepsilon_2)=(1,1)\} =\frac{N/2-1}{2(N-1)}$, implying
$$
P_2(1,1) = \frac{1}{2N} +\frac{(N-1)(N/2-1)}{2N(N-1)}=\frac{1}{4}.
$$
Similar calculations imply $P_2(1,-1) =  P_2(-1,1)=P_2(-1,-1)=\frac{1}{4}$, and so $P$ is pairwise independent.

Now, letting $\varepsilon =(\varepsilon_1,...\varepsilon_N)$ be a random variable with $law(\varepsilon) =P$, and noting that $ |\sum_{i=1}^N\varepsilon_i|^p$ is $0$ for points in the support of $P_b$ and $N$ for points in the support of $P_a$, we have

$$
\EE \left|\sum_{i=1}^N\varepsilon_i\right|^p =\frac{1}{N}N^p=N^{p-1}
$$
Noting that $||a||_2 =\sqrt{N}$, it follows that
$$
\left[\EE\left|\sum_{i=1}^N\varepsilon_i\right|^p\right] ^{1/p}=N^{1-1/p}=\sqrt{N} N^{1/2-1/p}=||a||_2N^{1/2-1/p}
$$
\end{proof}

We record next a partial uniqueness result on the $a$ and $\varepsilon$ giving equality in \eqref{theorem}.  The result establishes that the $|a_i|$ must all be equal, and that the support of $law(\varepsilon)$ must be as in the preceding proposition, up to permutation of the signs of the $a_i$.  If, in addition, $\varepsilon$ is exchangeable, its law is uniquely determined.
\begin{proposition}\label{prop: uniqueness}
Suppose $a'=(a_1',a_2',...,a_N') \in \mathbb{R}^N$ and $\varepsilon'=(\varepsilon_1',...\varepsilon_N')$ satisfy $\left(\EE \left|\sum_{i=1}^N a_i'\varepsilon_i'\right|^p\right)^{1/p}= N^{1/2-1/p} \|a'\|_2$, where $N$ is even, $\varepsilon'$ is pairwise independent and $p > 2$.  Then we must have $|a_i'|=c$ for some constant $c$, for all $i$.  Moreover,  almost surely, either:
\begin{enumerate}
\item[(a)]  $\varepsilon'_i =\sgn(a'_i)$ for all $i$,
\item[(b)] $\varepsilon'_i=-\sgn(a_i')$ for all $i$, 
\\or 
\item[(c)]  $\varepsilon'_i =\sgn(a_i')$ for exactly half of the $i$ (so that $\sum_{i=1}^Na_i'\epsilon_i'=0$).

\end{enumerate}
 Furthermore, we have 
$$
P'\{ \varepsilon'_i =\sgn(a_i') \text{ for all }i \} =P'\{ \varepsilon'_i =-\sgn(a_i') \text{ for all }i\}  =\frac{1}{2N},
$$
where $P'=law(\varepsilon')$.

Finally, if in addition $\varepsilon'$ is exchangeable and $N\geq 6$, we have $P'=P$, where $P=law(\varepsilon)$ from the preceding theorem.  %$\law(\varepsilon')  =\law(\varepsilon) $.
\end{proposition}
\begin{proof}
One can only have equality in \eqref{eqn:c-s} if each $|a_i|=c$ (to have equality in Cauchy-Schwartz), and if  either $\varepsilon'_i =\sgn(a'_i)$ for all $i$ or $\varepsilon'_i=-\sgn(a_i')$ for all $i$ with a positive probability (to have equality in the first inequality in\eqref{eqn:c-s}).  One can only have equality in  \eqref{eqn: holder} if the function $f$ takes on only two values (one of them being zero); taken together, then, we can only have equality in both \eqref{eqn: holder} and \eqref{eqn:c-s} if, almost surely, one of the conditions (a) - (c) holds.

%this implies that, almost surely,

% either $\epsilon_i =\sgn(a_i)$ for all $i$,  $\epsilon_i =-\sgn(a_i)$ for all $i$, or  $\sum_{i=1}^Na_i\epsilon_i=0$.  

Now, the fact that the $\epsilon_i'$ are Rademacher variables implies that $\EE(\sum_{i=1}^Na_i'\epsilon_i') =0$; as $\sum_{i=1}^Na_i'\epsilon_i'$ is either $0$, $Nc$ or $-Nc$ almost surely, we must have $\sum_{i=1}^Na_i'\epsilon_i' =\pm Nc$ with equal probability.   The constraint $\left(\EE \left|\sum_{i=1}^N a_i\varepsilon_i'\right|^p\right)^{1/p}= N^{1/2-1/p} \|a'\|_2$ then implies that each of these probabilities must be $\frac{1}{2N}$.

Turning to the final assertion, if $\epsilon'$ is exchangeable,  we claim that each of the $a'_i$ must then share the same sign.    To see this, note that if not, we can assume without loss of generality that $a'_1>0$ and $a'_2<0$.  As we have $P\{\varepsilon'_i=\sgn(a'_i)\text{ } \forall  i\} =\frac{1}{2N}$, symmetry implies that
 $$
 P\{\varepsilon'_1=-\sgn(a'_1)\text{, } \varepsilon'_2=-\sgn(a'_2) \text{ and }\varepsilon'_i=\sgn(a'_i) \text{ }\forall  i \geq 3\}=\frac{1}{2N} >0
 $$ 
 as well.  But  these values of $\epsilon'$, 
 %we have $\sum_{i=1}^Na_i'\varepsilon'_i =-a_1-a_2+\sum_{i=3}^Na_i'\varepsilon'_i=(N-4)c$, 
 do not satisfy any of the conditions (a) - (c) for $N \geq 6$.  This establishes the claim.

Therefore, we must have  $P(\epsilon_i'=1 \forall i)=P(\epsilon_i'=-1 \forall i) =\frac{1}{2N}$, and then symmetry implies that each $\epsilon'$ such that $\sum_{i=1}^N\epsilon_i'=0$ must have the same probability, which means that $law(\epsilon') =P$. 
\end{proof}
\begin{remark}
As  was pointed out to us by an anonymous referee on an earlier version of this manuscript, when $N=2^n$ is a power of $2$, an alternate construction can yield $C(N,2,p)=N^{1/2-1/p}$; let $(\bar \varepsilon_0,\bar \varepsilon_1,...,\bar \varepsilon_n)$ be a family of $n+1$ mutually independent Rademacher variables.  It is straightforward to see that the family $\{\epsilon_S:=\bar \epsilon_0\Pi_{i\in S}\bar \epsilon_i\}$ , where $S$ runs over the set of subsets of $\{1,2,....,n\}$, is a pairwise independent family of $2^n$ Rademacher random variables (here we have taken the convention that $\Pi_{i\in S}\epsilon_i=1$ when $S=\phi$ is empty).  It is straightforward to check that almost surely either each $\epsilon_S$ is $1$, each is $-1$ or exactly half are $1$, and the probability of each of the first two events is $\frac{1}{2^{n+1}}=\frac{1}{2N}$.  The claim follows.

This construction also shows that we do not have uniqueness up to $law(\epsilon)$ without the additional exchangeability condition, for $n \geq 3$.  To see this, note that the subset
$\{\bar\epsilon_0,\bar \epsilon_0 \bar \epsilon_1\bar \epsilon_2,\bar \epsilon_0\bar\epsilon_1\bar\epsilon_3,\bar \epsilon_0 \bar \epsilon_2\bar\epsilon_3 \}$ is not mutually independent (as the product of the first three elements is exactly the fourth).  On the other hand, the subset, $\{\bar\epsilon_0 ,\bar \epsilon_0 \bar \epsilon_1,\bar \epsilon_0\bar\epsilon_2,\bar \epsilon_0 \bar \epsilon_3 \}$ is mutually independent and so  $\{\epsilon_S\}_{S \in 2^{\{1,2,...,n\}}} $ is clearly not exchangeable.  
  Therefore, $law\{\epsilon_S \}\neq law(\epsilon)$.  
\end{remark}

\begin{remark}
It is straightforward to verify that the measure $P=law(\varepsilon)$ derived in  Theorem is in fact $3$-wise independent, and thus we immediately obtain analogues of the preceding results for $k=3$. In particular,

$$
C(N,p,3) = N^{1/2-1/p}.
$$

%$$
%C(N,p,3) = N^{1/4}.
%$$

It is clear that for $p>k \geq 4$, the estimate in Proposition \ref{prop: general bound} cannot be sharp; as the argument involved applying Holder's inequality to the function $f$, it can only be sharp if $f$ takes on at most two values (one of them being 0).  However, a direct calculation verifies that if $f$ takes on less than three values, then $\varepsilon$ cannot be $4$-wise independent.  Whether or not the growth rate in $N$ of $C(N,p,k)$  is \emph{proportional} to $N^{1/2 -k/2p}$ (that is, whether $C(N,p,k) = K(p,k)N^{1/2 -k/2p}$ for some constant $K(p,k) >1$) when $k \geq 4$ is an interesting open question.
%It turns out that $3$-dependent Rademacher random variables are also $2$-dependent. Thus, our result would hold for $3$-dependent Rademacher random variables:
%\[
%\left(\EE\left|\sum_{i=1}^Na_i \varepsilon_i\right|^p\right)^{1/p}=\left(\frac 34\right)^{1/p}\|I\|_{2}N^{1/2-1/p}.
%\]
\end{remark}
\begin{remark}
Generally speaking, one can identify exchangeable,   $k$-wise independent random variables $X_1, \ldots, X_N$ on $\mathbb{R}$ having equal fixed marginals $P_1=law(X_i)$ with permutation symmetric probability measures $P=law(X_1, \ldots, X_N)$ on $\mathbb{R}^N$ whose $k$-fold marginals are $\otimes ^k P_1$.  The set of measures satisfying these constraints is a convex set, and identifying the set of extremal points, or vertices, of this set is an interesting and nontrivial question.

Proposition \ref{prop: uniqueness} tells us that   the measure  $law(\epsilon)$ we construct is the unique maximizer of the linear functional $law(\epsilon)\mapsto \EE(\sum_{i=1}^N\epsilon_i)$ on the convex set of exchangeable, pairwise independent Rademacher probability measures.  As a consequence, we have therefore identified an extremal point of this set.

%It is easy to see upon inspection of the proof of Theorem 2.1 that the measure  $law(\epsilon)$ we construct is the unique maximizer of the linear functional $law(\epsilon)\mapsto \EE(\sum_{i=1}^N\epsilon_i)$ on the convex set of symmetric, pairwise independent Rademacher probability measures.  As a consequence of this proof, we have therefore identified an extremal %point of this set.

%Suppose that $x_1, \ldots, x_n$ are $k$-wise independent random variables. The constrains on the $k$-wise independent %distribution are finite list of equations and inequalities. The solution set is a polytope, which is very complicated, thus, it can be  %difficult to determine solutions. The vertices of such polytope is particularly interesting and non-trivial class-- $k$-wise independent %distributions with the small support. They often called weighted orthogonal arrays. Our Theorem 2.1. gives a class of vertices of %pairwise independent distributions in case when $x_1, \ldots, x_n$ pairwise independent Rademacher random variables.
\end{remark}

%\begin{remark}
%It would be nice to generalize our result to the case when $k\geq 4$.
%\end{remark}

%\begin{problem}
%Can one obtain similar results for the Marcinkiewicz–-Zygmund type inequalities?  One approach to this would be to first fix the linear program

%$$
%\sup
%$$
%\end{problem}

% ====================================
% =============== BIBLIOGRAPHY============
% ====================================


\begin{thebibliography}{50}

\bibitem{law}
D.~Andrews,
{\it Laws of Large Numbers for Dependent Non-Identically Distributed Random Variables
}
Econometric Theory, {\bf{4}},  (1988), 458--467.


\bibitem{D1}
Y.~Derriennic, A.~Klopotowski,
{\it Cinq variables al$\acute{\textmd{e}}$atoires, binaires,}
Institut Galil$\acute{\textmd{e}}$e, Universite Paris XIII, (1991), 1--38.
(1985), 109--117.


\bibitem{Etemadi81}
N. Etemadi
{\it An elementary proof of the strong law of large numbers} 
Z. Wahrsch. Verw. Gebiete 55 (1981), no.1, 119--122.


\bibitem{Garlin}
 D.~J.~H.~Garling,
{\it Inequalities: A Journey into Linear Analysis},
Cambridge University Press, Cambridge, 2007.





\bibitem{G}
 O.~Guedon,  P.~Nayar, T.~Tkocz,
{\it Concentration inequalities and geometry of convex bodies}, Extended notes of a course, Polish Academy of Sciences of Warsaw, to appear.
(http://perso-math.univ-mlv.fr/users/guedon.olivier/listepub.html)


\bibitem{H}
U.~Haagerup
{\it The best constants in the Khintchine inequality}
Studia Math. 70 (1981), no. 3, 231–283 (1982)

\bibitem{Khintchine23}
A. Khintchine
{\it \"Uber dyadische {B}r\"uche}
Math Z. 18 (1923), no. 1, 109--116.


\bibitem{Kah} J.P.~Kahane,
{\it Some random series of functions},  Second edition.
Cambridge Studies in Advanced Mathematics, 5,
Cambridge University Press, Cambridge, 1985.



 \bibitem{LT}
J.~Lindenstrauss, L.~Tzafriri,
{\it Classical Banach Spaces I and II,}
 Springer, 1996.



\bibitem{MS}
V.~D.~Milman, G.~Schechtman,
{\it Asymptotic theory of finite-dimensional normed spaces.}
With an appendix by M. Gromov. Lecture Notes in Math., 1200.
Springer-Verlag, Berlin, 1986.

\bibitem{P}
B.~Pass
{\it Multi-marginal optimal transport: theory and applications.}
To appear in  ESAIM: Math. Model. Numer. Anal.
\bibitem{PShir}
G.~Peskir, A.~N.~Shiryaev,
\emph{ The inequalities of Khintchine and expanding sphere of their action,}
Russian Math. Surveys
\textbf{50} 5 (1995), 849--904.



\bibitem{R1}
J.~Robertson,
{\it Independence and fair coin-tossing,}
Math. Scientist,
(1985), 109--117.



\bibitem{R2}
J.~Robertson,
{\it A two state pairwise independent stationary process for which $x_1x_3x_5$ is dependent,}
Sankhy$\bar{\textmd{a}}$, Series A, {\bf{50}}, (1988), 171--183.
(1985), 109--117.
















\end{thebibliography}
\end{document}